\documentclass[11pt,leqno]{amsart}
\usepackage{amsmath}
\usepackage{amsfonts}
\usepackage{amssymb}
\usepackage{graphicx}
\usepackage{hyperref}
\usepackage{color}
\usepackage[notref,notcite]{}
\newtheorem{theorem}{Theorem}

\newtheorem{definition}[theorem]{Definition}

\newtheorem{lemma}[theorem]{Lemma}

\newtheorem{proposition}[theorem]{Proposition}
\newtheorem{remark}[theorem]{Remark}

\numberwithin{equation}{section}
\numberwithin{theorem}{section}

\title[Capillary cmc hypersurfaces]{Area estimates for capillary cmc hypersurfaces with nonpositive Yamabe invariant}

 \author[L. F. Pessoa]{Leandro F. Pessoa} 
 \address{Departamento de Matemática, Universidade Federal do Piauí, 64049-550, Teresina - Piauí, Brazil}
 \email{leandropessoa@ufpi.edu.br}

 \author[E. V\'eras]{Erisvaldo V\'eras} 
 \address{Departamento de Matemática, Universidade Federal do Piauí,
 64049-550, Teresina - Piauí, Brazil}
 \email{erisvaldoveras35@gmail.com}

 \author[B. Vieira]{Bruno Vieira}
 \address{Centro de Educa\c c\~ao Aberta e \`a Dist\^ancia, Universidade Federal do Piau\'i-UFPI, 64049-550, Teresina - Piauí, Brazil}
 \email{bruno\_vmv@ufpi.edu.br}

\subjclass[2020]{Primary: 53C42, 53A10. Secondary: 53C21.}

\keywords{$cmc$ hypersurfaces, Capillary, Yamabe invariant}

\keywords{}

\begin{document}

\begin{abstract}
We prove area estimates for stable capillary $cmc$ (minimal) hypersurfaces $\Sigma$ with nonpositive Yamabe invariant that are properly immersed in a Riemannian $n$-dimensional manifold $M$ with scalar curvature $R^M$ and mean curvature of the boundary $H^{\partial M}$ bounded from below. We also prove a local rigidity result in the case $\Sigma$ is embedded and $\mathcal{J}$-energy-minimizing. In this case, we show that $M$ locally splits along $\Sigma$ and is isometric to $(-\varepsilon,\varepsilon)\times \Sigma, dt^2 + e^{-2Ht}g)$, where $g$ is Einstein, or Ricci flat, $H\geq 0$ and $\partial\Sigma$ is totally geodesic.
\end{abstract}

\maketitle

\section{Introduction}

The study of capillary surfaces can be traced back to the celebrated work of T. Young \cite{Young} that investigated the behavior of an incompressible liquid in a container in the absence of gravity. Variationally, capillary surfaces arise as critical points of an energy functional involving the area of the interface surface between the fluid and the air, the wet area on the boundary of the container, and the constant angle of contact between the interface surface and the boundary of the ambient space (cf. also \cite{Finn}). 

A special case occurs when the angle of contact is a right angle. In this case, the surface is called free boundary. This class of surfaces with boundary has been intensively studied in recent years due to its connection with the theory of closed minimal surfaces (see, e.g. \cite{Br,CFS,CFP,FPZ,FL,FS,FS1,Li1,Li2,LM,MNS} and references therein).

A great deal of work has been devoted to studying capillary stable surfaces, that is, surfaces that minimize the energy up to the second order. For instance, inspired by the work of J.C.C. Nitsche \cite{Ni}, A. Ros and E. Vergasta \cite{RV} proved that minimal free boundary hypersurfaces immersed in a ball of a space form are totally umbilical. The case of free boundary surfaces with constant mean curvature ($cmc$ for short) was settled combining efforts from \cite{RV} and \cite{Nunes} (see also \cite{Barbosa}). This result was finally solved for general capillary $cmc$ hypersurfaces in \cite{Wang-Xia}.

Based on the rigidity theorems for area-minimizing closed surfaces obtained in \cite{BBN,CG,MM,Nunes1}, L.C. Ambrozio \cite{Ambrozio} investigated the rigidity of area-minimizing free boundary surfaces embedded in a Riemannian three-dimensional manifold $M$, establishing an important relation between the scalar curvature of $M$ and the mean curvature of $\partial M$ to the topology (Euler characteristic) of stable minimal surfaces, which leads to a local splitting of the ambient manifold along the surface. The same problem was investigated by \cite{Li,Longa} for capillary $cmc$ surfaces.

For higher dimensions, rigidity results were obtained for closed hypersurfaces in \cite{Cai, Moraru25}, where the topological invariant considered was the so-called Yamabe invariant instead of the Euler characteristic. Indeed, for closed manifolds, the Yamabe invariant can be viewed as a generalization of the Euler characteristic because a manifold with nonpositive Yamabe invariant does not admit a metric of positive scalar curvature (cf. \cite[Lem.1.2]{Schoen89}). Furthermore, A. Barros and C. Cruz \cite{Barros} proved area estimates for compact, stable minimal hypersurfaces with free boundary in terms of the Yamabe invariant and the scalar curvature of the ambient. They also obtained rigidity and local splitting results for manifolds possessing an area-minimizing (or energy-minimizing) hypersurface. Recently, de Almeida and A. Mendes \cite{almeida_mendes}, among other results, extended the principal results of \cite{Barros} to the framework of free boundary marginally outer trapped surfaces (MOTS) within initial data sets with boundary.

Let $(M,\bar g)$ be an oriented Riemannian manifold with smooth boundary. In this note, we are interested in extending the rigidity results obtained in \cite{Barros} to the context of compact capillary $cmc$, or minimal, hypersurfaces $\Sigma$ with smooth boundary $\partial \Sigma$. 

Before stating our main results, we will introduce the Yamabe invariant for manifolds with boundary. Let $(\Sigma^{n-1},g)$ be a Riemannian manifold $n\geq 4$ with a nonempty boundary $\partial \Sigma$. Given $(a,b)\in\mathbb{R}\times\mathbb{R}-\{(0,0)\}$, we consider for all $\varphi \in C^\infty_+(\Sigma)$
\begin{equation*}
	Q_{g}^{a,b}(\varphi)=\frac{\displaystyle \int_{\Sigma}\left( \frac{4(n-2)}{n-3}||\nabla\varphi||_g^2+R_g\varphi^2\right)d\sigma +2\int_{\partial\Sigma}k_g\varphi^2d\sigma_{\partial\Sigma}}{\displaystyle  \left[ a \int_{\Sigma}\varphi^{\frac{2(n-1)}{n-3}}d\sigma +b\left( \int_{\partial\Sigma}\varphi^{\frac{2(n-2)}{n-3}}\right)^{\frac{n-1}{n-2}}  \right]^{\frac{n-3}{n-1}}},
\end{equation*}
where $k_g$ denotes the mean curvature of $\partial \Sigma$ and $R_g$ the scalar curvature of $\Sigma$. The Yamabe constant is defined by
$$Q_g^{a,b}(\Sigma,\partial \Sigma) = \inf_{\varphi \in C^\infty_+(\Sigma)} Q_g^{a,b}(\varphi),$$
with $(a,b) \in \{(1,0),(0,1)\}$. The Yamabe constant is invariant under conformal changes (see, for example, \cite{EscobarU,EscobarT}). Finding a critical point of the functional $Q_{g}^{1,0}$ corresponds to solving the Yamabe problem of finding a metric $\tilde{g} = u^{\frac{4}{n-3}}g$ with constant scalar curvature $R_{\tilde{g}}$ and with mean curvature of the boundary $k_{\tilde{g}}=0$. Similarly, the critical points of the functional $Q_{g}^{0,1}$ give solutions to the Yamabe problem of finding a metric $\tilde{g} = u^{\frac{4}{n-3}}g$ with constant mean curvature of the boundary $k_{\tilde{g}}$ and scalar curvature $R_{\tilde{g}}=0$.

Let $\mathcal{C}(\Sigma)$ be the space of all conformal classes $[g]$ of a metric $g$ in $\Sigma$. The Yamabe invariant of a compact manifold $(\Sigma,g)$ with boundary $\partial \Sigma$ is defined as
\begin{equation*}
	\sigma^{a,b}(\Sigma,\partial\Sigma)=\displaystyle\sup_{[g]\in\mathcal{C}(\Sigma)}\displaystyle\inf_{\varphi>0}Q^{a,b}_g(\varphi).
\end{equation*}

The first result of this note gives estimates from below for the area, and boundary area, of a compact capillary $cmc$ hypersurface. These estimates are natural extensions of the results obtained in \cite[Thm.1]{Barros}.

\begin{theorem}\label{teo1_intro}
	Let $M^n$ be a Riemannian manifold $(n\geq 4)$ with boundary $\partial M$. Let $\Sigma^{n-1}$ be a two-sided, stable, capillary $cmc$ hypersurface, properly immersed in $M$ with contact angle $\theta \in (0,\pi)$. 
	\begin{enumerate}
		\item [i)] Suppose that $H^{\partial M} + H\cos\theta \geq 0$, $\inf R^M + \frac{n}{n-1}H^2 < 0$ and $\sigma^{1,0}(\Sigma,\partial \Sigma)<0$. Then the area of $\Sigma$ satisfies
		\begin{equation}\label{vol_ineq_thm1}
			A(\Sigma)^{\frac{2}{n-1}}\geq\dfrac{Q^{1,0}_g(\Sigma,\partial \Sigma)}{\inf R^M +\frac{n}{n-1}H^2}\geq\dfrac{\sigma^{1,0}(\Sigma,\partial \Sigma)}{\inf R^M + \frac{n}{n-1}H^2}.
		\end{equation}
		\item [ii)] Suppose that  $R^M  +\frac{n}{n-1}H^2\geq 0$, $\inf  H^{\partial M} + H\cos\theta<0$ and $\sigma^{0,1}(\Sigma,\partial \Sigma)<0$. Then the area of $\partial\Sigma$ satisfies
		\begin{equation}\label{area_ineq_thm1}
			A(\partial\Sigma)^{\frac{1}{n-2}}\geq\frac{\sin \theta}{2}\dfrac{Q^{0,1}_g(\Sigma,\partial \Sigma)}{\inf H^{\partial M} + H\cos\theta}\geq \frac{\sin \theta}{2}\dfrac{\sigma^{0,1}(\Sigma,\partial \Sigma)}{\inf H^{\partial M} + H\cos\theta}.
		\end{equation}
	\end{enumerate}
\end{theorem}

\begin{remark}
Similar assumptions relating $H^{\partial M}$ to $H\cos\theta$, as well as $R^M$ to $H$, were recently considered in the study of the connection between the index and the geometry and topology of capillary $cmc$ surfaces, see \cite{hong_saturnino}. 
\end{remark}

The stability considered in Theorem \ref{teo1_intro} refers to the second variation of the following generalized $\mathcal{J}$-energy functional 
\begin{equation*}
	\mathcal{J}(t)=\mathcal{A}(t)-(\cos\theta)\mathcal{W}(t)+ H_0\mathcal{V}(t),
\end{equation*}
where $H_0 \in [0,\infty)$, $\mathcal{A}$ is the area functional, $\mathcal{V}$ is the volume functional and $\mathcal{W}$ is the wetting area functional  (see Section \ref{sec_prel}). The critical points for the $\mathcal{J}$-energy functional are capillary hypersurfaces with constant mean curvature equal to $H_0$ and contact angle $\theta \in (0,\pi)$.

In our second result, we investigate the rigidity of manifolds containing an energy-minimizing capillary hypersurface with respect to the $\mathcal{J}$-energy functional. 
The result should be compared with \cite[Thm.2,I)]{Barros}, and provides a counterpart for the $cmc$ case that was not considered in \cite[Thm.3]{Barros}. For the definition of an infinitesimally rigid hypersurface, we refer to Definition \ref{def2} in Section \ref{sec_prel}.

\begin{theorem}\label{teo2_intro}
Let $M^n$ be a Riemannian manifold $(n\geq 4)$ with boundary $\partial M$. Let $\Sigma^{n-1}$ be a $\mathcal{J}$-minimizing two-sided hypersurface, properly embedded in $M$ with contact angle $\theta \in (0,\pi)$, such that $H^{\partial M} + H\cos\theta \geq 0$. If $\inf R^M + \frac{n}{n-1}H^2 < 0$, $\sigma^{1,0}(\Sigma,\partial \Sigma)<0$ and the equality holds in \eqref{vol_ineq_thm1}, then $\Sigma$ is an infinitesimally rigid free-boundary minimal hypersurface and there is a neighborhood of $\Sigma$ in $M$ isometric to $(-\epsilon,\epsilon)\times\Sigma$ endowed with the metric $dt^{2}+g$, where $g$ is the induced metric on $\Sigma$ which is Einstein with totally geodesic boundary $\partial\Sigma$.
\end{theorem}

\begin{remark}
The rigidity on the contact angle $\theta$ has also been obtained in \cite[Thm.B]{Longa}. In case $\Sigma$ is minimal, we can also deduce the rigidity of the angle, namely $\theta = \pi/2$, under the assumptions of the item $ii)$ in Theorem \ref{teo1_intro} and assuming that the equality holds in \eqref{area_ineq_thm1}. Moreover, under these hypotheses, it is curious to observe that the splitting holds for $cmc$ hypersurfaces $\Sigma$ if we assume that $\partial \Sigma$ is locally area maximizing in $\partial M$. This is in contrast to the assumptions in \cite[Thm.7]{Ambrozio} and \cite[Thm.B]{Longa}.
\end{remark}

In our last result, we provide a capillary $cmc$ counterpart of \cite[Thm.2,II)\&Thm.3]{Barros}. 

\begin{theorem}\label{teo3_intro}
Let $M^n$ be a Riemannian manifold $(n\geq 4)$ with boundary $\partial M$. Let $\Sigma^{n-1}$ be a $\mathcal{J}$-minimizing two-sided hypersurface, properly embedded in $M$ with contact angle $\theta \in (0,\pi)$, such that $H^{\partial M} + H\cos\theta \geq 0$. If $R^M + \frac{n}{n-1}H^2 \geq 0$ and $\sigma^{1,0}(\Sigma,\partial \Sigma)\leq 0$, then $\Sigma$ is an infinitesimally rigid capillary $cmc$ hypersurface and there is a neighborhood of $\Sigma$ in $M$ isometric to $(-\epsilon,\epsilon)\times\Sigma$ endowed with the metric $dt^{2}+e^{-2Ht}g$, where $g$ is the induced metric in $\Sigma$ which is Ricci flat with a totally geodesic boundary $\partial\Sigma$.
\end{theorem}

\begin{remark}
It is clear that we can use $\sigma^{0,1}(\Sigma,\partial \Sigma)\leq 0$ instead of $\sigma^{1,0}(\Sigma,\partial \Sigma)\leq 0$ in Theorem \ref{teo3_intro}.
\end{remark}

\noindent \textbf{Acknowledgements.} The first author is funded by the Conselho Nacional de Desenvolvimento Cient\'ifico e Tecnol\'ogico - CNPq/Brazil, Grants: 422900/2021-4, 306543/2022-2. The second author is granted by the Funda\c c\~ao de Amparo \`a Pesquisa do Piau\'i - FAPEPI/Brazil, Scholarship: 029/2021.

\section{Preliminaries}\label{sec_prel}

Let $(M^n,\bar g)$ be a Riemannian manifold with boundary $\partial M$.  Let $(\Sigma^{n-1},g)$ be a smooth compact manifold with boundary $\partial \Sigma$, and let $\phi \colon \Sigma^{n-1} \rightarrow M^{n}$ be a smooth proper immersion, that is, $\phi(\Sigma)\cap \partial M = \phi(\partial \Sigma)$. In this note, we assume that $\Sigma$ is two-sided, in the sense that there is a unit normal vector field $N$ globally defined in $\Sigma$. We denote by $\nu$ the outward unit conormal for $\partial \Sigma$ in $\Sigma$, and by $\overline{N}$ the outward unit normal for $\partial M$. We consider $\overline{\nu}$ as the unit normal vector field defined along $\partial \Sigma$ in $\partial M$ such that the bases $\{N, \nu\}$ and $\{\overline{N}, \overline{\nu}\}$ have the same orientation in $(T\partial \Sigma)^{\perp}$.

We denote by $A$ and $H$ the second fundamental form and the (nonnormalized) mean curvature of $\Sigma$, respectively. We say that $\Sigma$ is a constant mean curvature hypersurface, cmc for short, if $H$ is constant along $\Sigma$. In this case, we will always consider $N$ in such a way that $H\geq 0$.

Given an immersion $\phi \colon \Sigma^{n-1} \rightarrow M^{n}$ we consider a proper smooth variation $\Phi \colon \Sigma \times (-\varepsilon,\varepsilon) \rightarrow M$ of $\phi$, that is, every map $\phi_t \colon \Sigma \rightarrow M$ defined by $\phi_t(x) = \Phi(x,t)$ is a proper immersion for all $t \in (-\varepsilon,\varepsilon)$ and $\phi_0 = \phi$. For a fixed variation $\Phi$ we introduce the area functional $\mathcal{A} \colon (-\varepsilon,\varepsilon) \rightarrow \mathbb{R}$ defined by
\begin{equation*}
\mathcal{A}(t)=\int_{\Sigma}dA_{\phi^{*}_{t}\bar g},
\end{equation*} 
where $dA_{\varphi^{*}_{t}\bar g}$ denotes the area element of $\Sigma$ with respect to the induced metric $\varphi^{*}_{t}\bar g$, and the volume functional $\mathcal{V} \colon (-\varepsilon,\varepsilon) \rightarrow \mathbb{R}$ by  
\begin{equation*}
\mathcal{V}(t)=\int_{\Sigma\times [0,t]}\Phi^{*}(dV_{\bar g}),
\end{equation*}
where  $dV_{\bar g}$ denotes the volume element of $M$. Notice that $\mathcal{V}$ expresses the volume between $\varphi_{t}(\Sigma)$ and $\varphi_{0}(\Sigma)$. We say that the variation $\Phi$ is volume-preserving if $\mathcal{V}(t)=0$ for every $t \in (-\epsilon, \epsilon)$.

Finally, we consider the wetting area functional $\mathcal{W}\colon (-\epsilon,\epsilon)\rightarrow \mathbb{R}$ defined by 
\begin{equation*}
\mathcal{W}(t)=\int_{\partial\Sigma\times [0,t]}\Phi^{*}(dA_{\partial M}),
\end{equation*} 
where $dA_{\partial M}$ is the area element of $\partial M$.

Fix two real numbers $\theta \in (0,\pi)$ and $H_0\in [0,\infty)$. We define the $\mathcal{J}$-energy functional $\mathcal{J}\colon (-\epsilon,\epsilon)\rightarrow \mathbb{R}$ by 
\begin{equation*}
	\mathcal{J}(t)=\mathcal{A}(t)-(\cos\theta)\mathcal{W}(t)+ H_0\mathcal{V}(t).
\end{equation*}
The formula for the first variation of the $\mathcal{J}$-energy functional is (see, e.g. \cite{BdCE,Ros})
\begin{eqnarray*}
\mathcal{J}'(0)&=&\int_{\Sigma}(H_0-H)f dA + \int_{\partial{\Sigma}}\bar g(\xi,\nu-(\cos\theta)\overline{\nu})dL,
\end{eqnarray*}
where $f=\bar g(\xi,N)$ and $\xi$ is the variation vector field associated to the variation $\Phi$.

An immersion $\varphi \colon \Sigma \rightarrow M$ is said to be a capillary $cmc$ hypersurface if it is a critical point of the energy functional $\mathcal{J}$ for every variation of $\varphi$. 
It is clear from the first variation formula that $\Sigma$ is a capillary $cmc$ hypersurface if and only if $\Sigma$ has constant mean curvature $H_0\geq 0$ and $\bar g(N,\overline{N}) = \cos \theta$ along $\partial \Sigma$. The angle $\theta$ is called the contact angle of the capillary hypersurface. In particular, when $\theta = \pi/2$ we say that $\Sigma$ is a free boundary cmc hypersurface.

For a capillary $cmc$ hypersurface with contact angle $\theta\in(0,\pi)$ the orthonormal bases $\{N,v\}$ and $\{\overline{N},\bar{\nu}\}$ are related by the following equations
\begin{equation}
	\left\{
	\begin{aligned}
		\overline{N} &= (\cos\theta)N + (\sin\theta)\nu \\
		\bar{\nu} &= -(\sin\theta)N + (\cos\theta)\nu.
	\end{aligned}
	\right.
\end{equation}
Furthermore, if $\Phi$ is a proper variation of $\varphi$, then the second variation of the energy functional $\mathcal{J}$ is given by
\begin{equation*}
\mathcal{J}''(0) =\int_{\Sigma}||\nabla f||^2-(Ric(N,N)+|A|^2)f^2-\int_{\partial\Sigma}qf^2,
\end{equation*}
where
\begin{equation*}
q=\frac{1}{\sin\theta}II^{\partial M}(\overline{\nu},\overline{\nu})+(\cot\theta) A(\nu,\nu),
\end{equation*}
and $II(v,w) = \bar g(\nabla_v \overline{N},w)$ is the second fundamental form of $\partial M$ with respect to $-\overline{N}$ (see \cite[Appendix]{Ros}).

We shall say that a capillary $cmc$ hypersurface $\varphi \colon \Sigma\rightarrow M$ is stable when $\mathcal{J}''(0) \geq 0$ for any proper variation $\Phi$ of $\varphi$.

We end this section by introducing the notion of infinitesimally rigid hypersurface. The minimal case was considered in \cite{Ambrozio}, and the $cmc$ case in \cite{Longa}. 

\begin{definition}\label{def2}
Let $\Sigma$ be a two-sided capillary hypersurface properly immersed in $M$. 
\begin{enumerate}
\item [i)] When $\Sigma$ is $cmc$, with $H>0$, it is said infinitesimally rigid if it is totally umbilical, the scalar curvature $R^{M}=\inf R^{M}$, $Ric(N,N)=-\frac{H^2}{n-1}$ along $\Sigma$, $H^{\partial M} =\inf H^{\partial M}$ along $\partial\Sigma$ and $\Sigma$ is Einstein with respect to the induced metric.
\item [ii)] When $\Sigma$ is minimal it is said infinitesimally rigid if it is totally geodesic, the scalar curvature $R^{M} = \inf R^{M}$, $Ric(N,N)=0$ along $\Sigma$, $H^{\partial M}=\inf H^{\partial M}$ along $\partial\Sigma$ and $\Sigma$ is Einstein with respect to the induced metric.
\end{enumerate}
\end{definition}

For an infinitesimally rigid capillary hypersurface $\Sigma$ one can prove the existence of a local foliation around $\Sigma$ by $cmc$ capillary hypersurfaces $\Sigma_t$. For a proof of the next result, we refer to \cite[Prop.3]{Barros} or \cite[Prop.3.2]{Longa}.

\begin{proposition}[\cite{Barros,Longa}]\label{Prop7}
Let $M^n$ be a Riemannian manifold with nonempty boundary such that $H^{\partial M}$ and $R^{M}$ are limited from below. Assume that $M^n$ contains a properly embedded capillary $cmc$  hypersurface $\Sigma$ with contact angle $\theta \in (0,\pi)$. If $\Sigma$ is infinitesimally rigid, then there exists a small neighborhood of $\Sigma$ in $M$ foliated by capillary $cmc$ hypersurfaces $\Sigma_t$ with contact angle $\theta$ and isotopic to $\Sigma$.
\end{proposition}

\section{Area estimates}\label{sect_1}

In this section, we will prove the area estimates for stable capillary $cmc$ hypersurfaces $\Sigma$ stated in Theorem \ref{teo1_intro}. 

We begin with a useful lemma which will be crucial also for the proof of the other results. In what follows, for ease of notation, we will use $H$ instead of $H_0$ to denote the constant mean curvature of $\Sigma$.

\begin{lemma}[Lemma 3.1 in \cite{Longa}]\label{lema1}
Let $M^{n}$ be a Riemannian manifold with nonempty boundary and let $\Sigma^{n-1}$ be a two-sided capillary $cmc$ hypersurface immersed in $M$ with contact angle $\theta\in(0,\pi)$ and mean curvature equal to $H\geq 0$. Then
	\begin{equation*}
		II^{\partial M}(\overline{\nu},\overline{\nu})+(\cos \theta) A(\nu,\nu)+(\sin \theta) k_{g}=H^{\partial M}+H \cos\theta,
	\end{equation*}
 where $k_g$ is the mean curvature of $\partial\Sigma$ and $H^{\partial M}$ is the mean curvature of $\partial M$.
\end{lemma}
\begin{proof} Let $\{e_{i}\}^{n-2}_{i=1}$ be an orthonormal basis for  $T\partial \Sigma$. Since $\Sigma^{n-1}$ is $cmc$, we obtain
	\begin{equation*}
		A(\nu,\nu)+\sum_{i=1}^{n-2}A(e_i,e_i)=H.
	\end{equation*}
	Thus, 
	\begin{eqnarray*}
		(\cos\theta) A(\nu,\nu) +(\sin\theta) k_g - H\cos \theta&=&\cos\theta \left(H-\sum_{i=1}^{n-2}A(e_i,e_i)\right) +(\sin\theta) k_g - H\cos \theta \\
		&=& -\cos \theta \sum_{i=1}^{n-2}A(e_i,e_i)+(\sin\theta) k_g\\
		&=& \sum_{i=1}^{n-2}\bar g(-\nabla_{e_i}e_i,\overline{N}) \\
		&=& \sum_{i=1}^{n-2} II^{\partial M}(e_{i},e_{i}).
	\end{eqnarray*}
However, $\{\overline{\nu},e_1,\ldots,e_{n-2}\}$ is an orthonormal referential for $T\partial M$, and we have
	\begin{eqnarray*}
		II^{\partial M}(\overline{\nu},\overline{\nu})+(\cos \theta) A(\nu,\nu)+(\sin\ \theta) k_{g}-H \cos\theta&=&II^{\partial M}(\overline{\nu},\overline{\nu})+\sum_{i=1}^{n-2}II^{\partial M}(e_i,e_i)\\
		&=& H^{\partial M}.
	\end{eqnarray*}
Therefore,
	\begin{equation*}
		II^{\partial M}(\overline{\nu},\overline{\nu})+(\cos \theta) A(\nu,\nu)+(\sin \theta) k_{g}=H^{\partial M}+H \cos\theta.
	\end{equation*}
\end{proof}

We will restate Theorem \ref{teo1_intro} for the reader's convenience.

\begin{theorem}\label{teo1}
	Let $M^n$ be a Riemannian manifold $(n\geq 4)$ with boundary $\partial M$. Let $\Sigma^{n-1}$ be a two-sided, stable, capillary $cmc$ hypersurface, properly immersed in $M$ with contact angle $\theta \in (0,\pi)$. 
	\begin{enumerate}
		\item [i)] Suppose that $H^{\partial M} + H\cos\theta \geq 0$, $\inf R^M + \frac{n}{n-1}H^2 < 0$ and $\sigma^{1,0}(\Sigma,\partial \Sigma)<0$. Then the area of $\Sigma$ satisfies
		\begin{equation}\label{vol_ineq_thm11}
			A(\Sigma)^{\frac{2}{n-1}}\geq\dfrac{Q^{1,0}_g(\Sigma,\partial \Sigma)}{\inf R^M +\frac{n}{n-1}H^2}\geq\dfrac{\sigma^{1,0}(\Sigma,\partial \Sigma)}{\inf R^M + \frac{n}{n-1}H^2}.
		\end{equation}
		\item [ii)] Suppose that  $R^M  +\frac{n}{n-1}H^2\geq 0$, $\inf  H^{\partial M} + H\cos\theta<0$ and $\sigma^{0,1}(\Sigma,\partial \Sigma)<0$. Then the area of $\partial\Sigma$ satisfies
		\begin{equation*}
			A(\partial\Sigma)^{\frac{1}{n-2}}\geq\frac{\sin \theta}{2}\dfrac{Q^{0,1}_g(\Sigma,\partial \Sigma)}{\inf H^{\partial M} + H\cos\theta}\geq \frac{\sin \theta}{2}\dfrac{\sigma^{0,1}(\Sigma,\partial \Sigma)}{\inf H^{\partial M} + H\cos\theta}.
		\end{equation*}
	\end{enumerate}
\end{theorem}
\begin{proof} Let us first prove the case $i)$. By Newton inequality and the Gauss equation we have
	\begin{equation*}
		Ric(N,N)+|A|^2\geq \frac{1}{2}(R^M-R_g+\frac{n}{n-1}H^2).
	\end{equation*}
Substituting the above inequality and Lemma \ref{lema1} in the stability condition for $\Sigma^{n-1}$ we obtain  
	\begin{eqnarray}\label{main_ineq_vol}
		0&\leq& \int_{\Sigma}\left(2||\nabla\varphi||^2+R_g\varphi^2\right)+\int_{\partial{\Sigma}}2k_{g}\varphi^2-\int_{\Sigma}\left(\frac{n}{n-1}H^2 + R^M\right)\varphi^2 \\[0.2cm]
  && -\frac{2}{\sin\theta}\int_{\partial\Sigma}\left(H^{\partial M} + H\cos\theta\right)\varphi^2,\nonumber
	\end{eqnarray}
where  $\varphi\in C^{\infty}_{+}(\Sigma)$, and $H^{\partial M}$ is the mean curvature of $\partial M$. 

Recalling that $a_n=\frac{4(n-2)}{n-3}>2$ for $n\geq4$ and, by assumption, $H^{\partial M} + H\cos\theta \geq 0$,  it follows that
\begin{equation*}\label{eq2.1}
	\left(\inf R^M+\frac{n}{n-1}H^2\right) \int_{\Sigma}\varphi^2\leq\int_{\Sigma}(a_{n}||\nabla\varphi||^2+R_g\varphi^2)+\int_{\partial{\Sigma}}2k_{g}\varphi^2.
\end{equation*}
Since $\inf R^{M}+\frac{n}{n-1}H^2<0$ we can use H\"older’s inequality to deduce 
\begin{eqnarray*}
	\left(\inf R^M+\frac{n}{n-1}H^2\right) A(\Sigma)^{\frac{2}{n-1}}
	&\leq& Q^{1,0}_g(\varphi).
\end{eqnarray*} 
Therefore, we conclude that
\begin{equation*}
	A(\Sigma)^{\frac{2}{n-1}}\geq\dfrac{Q^{1,0}_g(\Sigma,\partial \Sigma)}{\inf R^M+\frac{n}{n-1}H^{2}}\geq\dfrac{\sigma^{1,0}(\Sigma,\partial \Sigma)}{\inf R^M+\frac{n}{n-1}H^{2}}.
\end{equation*}

Item $ii)$ follows similarly. Now, inequality \eqref{main_ineq_vol} coupled with the assumptions yield 
\begin{eqnarray*}
	\frac{2}{\sin \theta}(\inf\ H^{\partial M} + H\cos\theta)\int_{\partial\Sigma}\varphi^{2}
	&\leq& \int_{\Sigma}(a_{n}||\nabla\varphi||^2+R_g\varphi^2)+\int_{\partial{\Sigma}}2k_{g}\varphi^2.
\end{eqnarray*}
Then, from H\"older’s inequality we deduce 
\begin{equation*}
	A(\partial\Sigma)^{\frac{1}{n-2}}\geq\frac{\sin \theta}{2}\dfrac{Q^{0,1}_g(\Sigma,\partial \Sigma)}{(\inf\ H^{\partial M}+H\cos\theta)}\geq \frac{\sin \theta}{2}\dfrac{\sigma^{0,1}(\Sigma,\partial \Sigma)}{(\inf\ H^{\partial M}+H\cos\theta)}.
\end{equation*}
This completes the proof.
\end{proof}

\section{Rigidity of capillary hypersurfaces}

Our main goal in this section is to prove the local splitting theorems for capillary $cmc$ hypersurfaces, namely Theorems \ref{teo2_intro} and \ref{teo3_intro} stated in the Introduction. Let $M^n$ be a Riemannian manifold $(n\geq 4)$ with nonempty boundary $\partial M$. In this section, we will always consider $\Sigma^{n-1}\subset M^n$ as a capillary $cmc$, compact, two-sided, properly immersed hypersurface with constant mean curvature $H\geq 0$ and contact angle $\theta \in (0,\pi)$. We will first establish the infinitesimal rigidity of $\Sigma$ under the conditions of Theorem \ref{teo2_intro}.

\begin{proposition}\label{prop5}
 Assume that $H^{\partial M} + H\cos\theta\geq 0$, $\sigma^{1,0}(\Sigma,\partial \Sigma)<0$, $\inf R^M + \frac{n}{n-1}H^2< 0$ and the equality holds in \eqref{vol_ineq_thm11}, that is,
\begin{equation*}
A(\Sigma)^{\frac{2}{n-1}} = \frac{\sigma^{1,0}(\Sigma,\partial \Sigma)}{\inf R^M+ \frac{n}{n-1}H^2}.
\end{equation*} 
Then $\Sigma$ is infinitesimally rigid, $II^{\partial M}(\overline{\nu},\overline{\nu}) + A(\nu,\nu)\cos\theta=0$ along $\partial \Sigma$ and the induced metric on $\Sigma$ is Einstein with totally geodesic boundary $\partial\Sigma$. 
\end{proposition}
\begin{proof}
The hypothesis $\sigma^{1,0}(\Sigma,\partial \Sigma)<0 < Q^{1,0}_g(\mathbb{S}^{n-1}_+,\partial\mathbb{S}^{n-1}_+)$ guarantees the existence of a solution to the Yamabe problem (see \cite{EscobarT}), that is, there exists $\varphi_{min}>0$ for which the infimum in $Q_{g}^{1,0}(\Sigma,\partial \Sigma)$ is achieved. In this case, all inequalities in the proof of Theorem \ref{teo1} item $i)$ become equalities, and thus, we have that $R^M = \inf R^M$, $\Sigma$ is totally umbilical, $H^{\partial M} + H\cos\theta = 0$ along $\partial\Sigma$ and $\varphi_{min}$ must be constant, because $a_n>2$. As in the proof of \cite[Prop.1]{Barros}, the stability condition on $\Sigma$ gives that zero is the first eigenvalue and $\varphi_{min}$ is an eigenfunction of the Robin-type problem 
	\begin{equation}\label{eq2.6}
		\left\{
		\begin{array}{rl}
			-\Delta_{\Sigma}\phi-(Ric(N,N)+|A|^2)\phi&=\lambda\phi \ \ \text{em} \ \ \Sigma,  \\[0.2cm] 
			\dfrac{\partial\phi}{\partial\nu}&= q\phi \ \ \text{em} \ \ \partial\Sigma,
		\end{array}
		\right.
	\end{equation}
	with
	\begin{eqnarray*}
		q=\dfrac{1}{\sin\theta}II^{\partial M}(\overline{\nu},\overline{\nu})+(\cot\theta) A(\nu,\nu).
	\end{eqnarray*}
Hence,	$Ric(N,N)= - \vert A\vert^2 = - \frac{H^2}{n-1}$ and $II^{\partial M}(\overline{\nu},\overline{\nu}) + A(\nu,\nu)\cos\theta=0$ along $\partial\Sigma$. It follows easily from Lemma \ref{lema1} that $k_{g} = 0$ ($\partial\Sigma$ is a minimal hypersurface). We refer to the proof of \cite[Prop.1]{Barros} for the proof that the metric $g$ is Einstein with totally geodesic boundary.
\end{proof}

We are ready to prove Theorem \ref{teo2_intro} which is restated below.

\begin{theorem}
Let $M^n$ be a Riemannian manifold $(n\geq 4)$ with boundary $\partial M$. Let $\Sigma^{n-1}$ be a $\mathcal{J}$-minimizing two-sided hypersurface, properly embedded in $M$ with contact angle $\theta \in (0,\pi)$, such that $H^{\partial M} + H\cos\theta \geq 0$. If $\inf R^M + \frac{n}{n-1}H^2 < 0$, $\sigma^{1,0}(\Sigma,\partial \Sigma)<0$ and the equality holds in \eqref{vol_ineq_thm1}, then $\Sigma$ is an infinitesimally rigid free-boundary minimal hypersurface and there is a neighborhood of $\Sigma$ in $M$ isometric to $(-\epsilon,\epsilon)\times\Sigma$ endowed with the metric $dt^{2}+g$, where $g$ is the induced metric on $\Sigma$ which is Einstein with totally geodesic boundary $\partial\Sigma$.
\end{theorem}
\begin{proof}
Proposition \ref{prop5} shows that $\Sigma$ is infinitesimally rigid, and then, by Proposition \ref{Prop7} there is a local foliation by $cmc$ capillary hypersurfaces $\Sigma_t$ for $t \in (-\varepsilon,\varepsilon)$, with $\varepsilon>0$ sufficiently small.

We recall that the lapse function $\rho_t$ satisfies (see \cite{Li})
\begin{equation}\label{lapse_cor}
	\left\{
	\begin{array}{rcl}
		H'(t)&=&\Delta_{\Sigma}\rho_t+(Ric(N_{t},N_{t})+|A_{t}|^2)\rho_t \ \ \text{em} \ \ \Sigma_{t}, \\[0.2cm]
		\dfrac{\partial \rho_t}{\partial \nu_{t}}&=&\left(\dfrac{1}{\sin\theta}II^{\partial M}(\overline{\nu}_{t},\overline{\nu}_{t})+(\cot\theta) A_{t}(\nu_{t},\nu_{t})\right)\rho_t \ \ \text{em} \ \ \partial\Sigma_{t}.
	\end{array}
	\right.
\end{equation}
By Gauss equation and Newton inequality, we can rewrite the first equation in \eqref{lapse_cor} as
	\begin{eqnarray}\label{2.7}
	2H'(t)\rho^{-1}_t\geq 2\rho^{-1}_t\Delta_{\Sigma}\rho_t+R^M_{t}-R_{t}+\frac{n}{n-1}H(t)^2.
\end{eqnarray}

Denote by $g_{t}$ the induced metric of $\Sigma_{t}$. As in the proof of Proposition \ref{prop5}, for each $t\in(-\epsilon,\epsilon)$ we solve the Yamabe problem with boundary to guarantee the existence of a metric $\tilde{g}_{t}= u_{t}^{\frac{4}{n-3}}g_{t}$ with constant scalar curvature equal to $\inf R^M < 0$ and whose boundary is a minimal hypersurface. The function $u_{t}$ is a positive function on $\Sigma_{t}$ 
such that the Yamabe constant is achieved.

Multiplying \eqref{2.7} by $u_t^2$ and integrating along $\Sigma_t$ it becomes
	\begin{eqnarray*}
	2\int_{\Sigma_{t}}H'(t)\rho^{-1}_tu^2_{t} &\geq&2\int_{\Sigma_{t}}\rho^{-1}_tu^2_{t}\Delta_{\Sigma}\rho_t+ \int_{\Sigma_{t}}\left(R^M + \frac{n}{n-1}H(t)^2\right)u^2_{t}-\int_{\Sigma_{t}}u^2_{t}R_{t}.
\end{eqnarray*}
Since $\Sigma_t$ is $cmc$, the divergence theorem combined with Young inequality will lead to 
\begin{eqnarray*}
	2H'(t)\int_{\Sigma_{t}}\rho^{-1}_tu^2_{t} &\geq& 2\left(\int_{\Sigma_{t}}||\nabla u_t||^2 + \int_{\partial \Sigma_{t}}\rho^{-1}_tu^2_{t}\dfrac{\partial\rho_t}{\partial\nu_{t}}\right)-\int_{\Sigma_{t}}u^2_{t}R_{t} \\[0.2cm]
 && +\left(\inf R^M + \frac{n}{n-1}H(t)^2\right)\int_{\Sigma_{t}}u^2_{t}.
\end{eqnarray*}
Substituting the second line of \eqref{lapse_cor} and applying Lemma \ref{lema1} we obtain
\begin{multline*}
2H'(t)\int_{\Sigma_{t}}\rho^{-1}_tu^2_{t} \geq	 -\left[\int_{\Sigma_{t}}(2||\nabla u_{t}||^2+u^2_{t}R_{t})+2\int_{\partial \Sigma_{t}}u^2_{t}k_{t}\right]\\[0.2cm]
+\frac{2}{\sin\theta}\int_{\partial \Sigma_{t}}\left(H^{\partial M}_{t} + H(t)\cos\theta\right)u^2_{t}+\left(\inf R^M + \frac{n}{n-1}H(t)^2\right)\int_{\Sigma_{t}}u^2_{t}.
\end{multline*}
Since $H^{\partial M}\geq -H\cos\theta$ and $a_{n}=\frac{4(n-2)}{n-3}>2$ for $n\geq4$, we can rewrite the above inequality as
\begin{equation}\label{ineq_teo3}
H'(t)\psi_1(t)  \geq -Q^{1,0}(u_t)+\psi_2(t) \left(H(t)-H\right)\cot\theta + \psi_3(t)\left(\inf R^M + \frac{n}{n-1}H(t)^2\right),
\end{equation}
where  
\begin{equation}\label{psi_12}
\psi_1(t)=\frac{2\displaystyle\int_{\Sigma_{t}}\rho^{-1}_tu^2_{t}}{\left(\displaystyle\int_{\Sigma_{t}}u^{\frac{2(n-1)}{n-3}}_{t}\right)^{\frac{n-3}{n-1}}}, \qquad \psi_2(t)=\frac{2\displaystyle\int_{\partial \Sigma_{t}}u^2_{t}}{\left(\displaystyle\int_{\Sigma_{t}}u^{\frac{2(n-1)}{n-3}}_{t}\right)^{\frac{n-3}{n-1}}},
\end{equation}
and 
$$\psi_3(t)=\frac{\displaystyle\int_{\Sigma_{t}}u^2_{t}}{\left(\displaystyle\int_{\Sigma_{t}}u^{\frac{2(n-1)}{n-3}}_{t}\right)^{\frac{n-3}{n-1}}} \leq A(\Sigma_t)^{\frac{2}{n-1}}.$$

By continuity we can assume $\inf R^M + \frac{n}{n-1}H(t)^2 < 0$. Now since the equality holds in \eqref{vol_ineq_thm11}, we arrive at
\begin{eqnarray*}
H'(t)\psi_1(t)&\geq& -\sigma^{1,0}(\Sigma,\partial\Sigma)+ \psi_2(t) \left(H(t)-H\right) \cot\theta   \\[0.2cm]
&&+ \left(\inf R^M + \frac{n}{n-1}H(t)^2\right) A(\Sigma_{t})^\frac{2}{n-1}   \\[0.2cm]
&=& - \left(\inf R^M + \frac{n}{n-1}H^2\right)A(\Sigma)^\frac{2}{n-1}  + \psi_2(t) \left(H(t)-H\right)\cot\theta\\[0.2cm] 
&& +\left(\inf R^M + \frac{n}{n-1}H(t)^2\right)A(\Sigma_{t})^\frac{2}{n-1}\\[0.2cm]
&=& \left(\inf R^M + \frac{n}{n-1}H^2\right)\left(A(\Sigma_{t})^\frac{2}{n-1}-A(\Sigma)^\frac{2}{n-1}\right) \\[0.2cm]
&&+\psi_2(t) \left(H(t)-H\right)\cot\theta+ \frac{n}{n-1}\left(H(t)^2 - H^2\right)A(\Sigma_{t})^\frac{2}{n-1}.
\end{eqnarray*}
Recalling that
 \begin{multline*}   
A(\Sigma_{t})^\frac{2}{n-1}-A(\Sigma)^\frac{2}{n-1}=\frac{2}{n-1}\int_{0}^{t}\left(\frac{d}{ds}A(\Sigma_{s})\right)A(\Sigma_{s})^{\frac{3-n}{n-1}}ds\\[0.2cm]
= -\frac{2}{n-1}\int_{0}^{t} A(\Sigma_{s})^{\frac{3-n}{n-1}}\left(H(s)\left(\int_{\Sigma}\rho(s)\right)+\cot\theta\left(\int_{\partial\Sigma}\rho(s)\right)\right),
 \end{multline*}
we can rewrite our estimate as
\begin{multline*}
H'(t)\psi_1(t) \geq -\frac{2}{n-1}\left(\inf R^M + \frac{n}{n-1}H^2\right)\int_{0}^{t}(H(s)\xi(s) + \eta(s)\cot\theta) ds  \\[0.2cm]
+\psi_2(t)(H(t)-H)\cot\theta + \frac{n}{n-1}(H(t)^2 - H^2) A(\Sigma_{t})^\frac{2}{n-1}.
\end{multline*}
where
$$\xi(s)=A(\Sigma_{s})^{\frac{3-n}{n-1}}\int_{\Sigma}\rho_s\quad\text{and}\quad \eta(s)=A(\Sigma_{s})^{\frac{3-n}{n-1}}\int_{\partial\Sigma}\rho_s.$$

Now, since the lapse function $\rho_t$ can be assumed positive we may conclude that
\begin{equation}\label{ineq_H_teo2}   
H'(t)\geq \left[R_0(t) \int_{0}^{t}\eta(s) ds + R_1(t)(H(t)-H)\right]\cot\theta + R_2(t)(H(t)^2 - H^2) ,
\end{equation}
with 
$$R_0(t) = -\frac{2}{(n-1)\psi_1(t)}\left(\inf R^M + \frac{n}{n-1}H^2\right), \qquad R_1(t) = \frac{\psi_2(t)}{\psi_1(t)}, $$
and
$$R_2(t) = \frac{n}{(n-1)\psi_1(t)}A(\Sigma_{t})^\frac{2}{n-1}.$$

Let us suppose that $H>0$. We first observe that $H'(0)\geq 0$. If $H'(0)>0$, then $H(t)>0$ for small positive $t$ and by the first variation formula for the $\mathcal{J}$-functional we would have 
\begin{equation}
	J(t)-J(0)=\int_{0}^{t}\left(H-H(s)\right)\left(\int_{\Sigma}\rho(s)\right) ds <0,
\end{equation}
which contradicts the fact that $\Sigma$ is $\mathcal{J}$-minimizing. Thus, $H'(0)=0$. We set
$$g(t) = \left[R_0(t) \int_{0}^{t}\eta(s) ds + R_1(t)(H(t)-H)\right]\cot\theta + R_2(t)(H(t)^2 - H^2) .$$
By definition, it is easy to see that 
$$H''(0) = g'(0) = R_0(0)\eta(0)\cot\theta.$$
If $H''(0)>0$, then $H'(t)>0$ for $t>0$ sufficiently small, that is, $H(t)$ would be increasing which contradicts the $\mathcal{J}$-minimizing assumption. Similarly, if $H''(0)<0$ we then deduce that $H'(t)>0$ and $H(t)<0$ for $t<0$ sufficiently small, and we arrive at the same contradiction. Therefore, $\Sigma$ must be a free boundary hypersurface. Therefore, inequality \eqref{ineq_H_teo2} reduces to
\begin{eqnarray*}   
H'(t)&\geq& R_2(t)(H(t)^2 - H^2). 
\end{eqnarray*}
Applying Gronwall Lemma (see \cite{gronwall}) to the function $u(t) = H-H(t)$ we obtain
$$H(t) - H\geq (H(s) - H)\exp\left(\int_{s}^{t}R_2(\tau)(H(\tau)+H)d\tau\right),$$ 
for $s<t$. However, $H(0)=H$ forces that $H(t)\geq H$ for $t>0$ and $H(t)\leq H$ for $t<0$, that is, $J(t)\leq J(0)$, for all $t\in(-\varepsilon,\varepsilon)$ and $J(t)= J(0)$ for all $ t\in(-\varepsilon,\varepsilon)$ due to the fact that $\Sigma$ is $\mathcal{J}$-minimizing. Therefore, $H(t) \equiv H$ and each $\Sigma_{t}$ is a cmc capillary stable hypersurface. 
On the other hand, applying Theorem \ref{teo1_intro} to $\Sigma_t$ we must have
$$0 \leq A(\Sigma_t) - A(\Sigma) = \int_0^t \frac{d}{ds}A(\Sigma_{s})ds = - H\int_0^t \left(\int_{\Sigma}\rho(s)\right) ds,$$
which means that $H\leq 0$, a contradiction to the assumption that $H>0$. If $H=0$ we also conclude that $\Sigma$ is free boundary, and thus, the result follows directly from \cite[Thm.2]{Barros}.
\end{proof}

We now establish the infinitesimal rigidity of $\Sigma$ under the assumptions of Theorem \ref{teo3_intro}.

\begin{proposition}\label{prop6}
If $H^{\partial M} + H\cos\theta\geq 0$, $\sigma^{1,0}(\Sigma,\partial \Sigma)\leq 0$ and $R^M + \frac{n}{n-1}H^2 \geq 0$, then $\Sigma$ is infinitesimally rigid, $II^{\partial M}(\overline{\nu},\overline{\nu}) + A(\nu,\nu)\cos\theta=0$ along $\partial \Sigma$, $\sigma^{1,0}(\Sigma,\partial\Sigma) = 0$, and $\partial\Sigma$ has zero mean curvature. Moreover, $\Sigma$ is Ricci flat with totally geodesic boundary with respect to the induced metric.
\end{proposition}
\begin{proof}
Substituting the assumptions in inequality \eqref{main_ineq_vol} we obtain
\begin{eqnarray*}
0 &\leq&\int_{\Sigma}(a_{n}||\nabla\varphi||^2+R_g\varphi^2)+\int_{\partial{\Sigma}}2k_{g}\varphi^2.
	\end{eqnarray*}
It is easy to conclude that
\begin{eqnarray*}
0 & \leq & Q_{g}^{1,0}(\Sigma,\partial \Sigma)\leq\sigma^{1,0}(\Sigma,\partial \Sigma)\leq 0.
\end{eqnarray*}
Again, using a solution to the Yamabe problem as in the proof of Proposition \ref{prop5}, we will conclude that $R^{M}=- \frac{n}{n-1}H^2$, $H^{\partial M}= - H\cos\theta$ along $\partial\Sigma$, $\Sigma$ is totally umbilical, $Ric(N,N)= - \frac{H^2}{n-1}$ in $\Sigma$, $II^{\partial M}(\overline{\nu},\overline{\nu}) +A(\nu,\nu)\cos\theta=0$ along $\partial \Sigma$ and $k_{g}=0$. On the other hand, by Gauss equation
	\begin{equation*}
		R_{g}=R^{M}+H^{2}-|A|^{2}-2Ric(N,N)=0.
	\end{equation*}
As in Proposition \ref{prop5}, the metric induced in $\Sigma$ is Einstein with a totally geodesic boundary. Thus, it must be Ricci flat.
\end{proof}

We are ready to prove Theorem \ref{teo3_intro} which is restated below.

\begin{theorem}
Let $M^n$ be a Riemannian manifold $(n\geq 4)$ with boundary $\partial M$. Let $\Sigma^{n-1}$ be a $\mathcal{J}$-minimizing two-sided hypersurface, properly embedded in $M$ with contact angle $\theta \in (0,\pi)$, such that  $H^{\partial M} + H\cos\theta \geq 0$. If $R^M + \frac{n}{n-1}H^2 \geq 0$ and $\sigma^{1,0}(\Sigma,\partial \Sigma)\leq 0$, then $\Sigma$ is an infinitesimally rigid capillary $cmc$ hypersurface and there is a neighborhood of $\Sigma$ in $M$ isometric to $(-\epsilon,\epsilon)\times\Sigma$ endowed with the metric $dt^{2}+e^{-2Ht}g$, where $g$ is the induced metric on $\Sigma$ which is Ricci flat with totally geodesic boundary $\partial\Sigma$.
\end{theorem}
\begin{proof}
The proof is analogous to the proof of Theorem \ref{teo2_intro}. In this case, Proposition \ref{prop6} gives that $\Sigma$ is infinitesimally rigid and we can consider the cmc foliation $\Sigma_t$ given by Proposition \ref{Prop7}. Since $\inf R^M + \frac{n}{n-1}H^2\geq 0$ and $\sigma^{1,0}(\Sigma,\partial \Sigma)\leq 0$, inequality  \eqref{ineq_teo3} reduces to
\begin{eqnarray*}
H'(t)\psi_1(t)  \geq \psi_2(t) \left(H(t)-H\right)\cot\theta 
\end{eqnarray*}
with $\psi_1$ and $\psi_2$ defined in \eqref{psi_12}.
Applying Gronwall Lemma (see \cite{gronwall}) to the function $u(t) = H-H(t)$ we obtain
$$H(t)-H\geq (H(s)-H)\exp\left(\cot\theta\int_{s}^{t}\frac{\psi_2(\tau)}{\psi_1(\tau)}d\tau\right),$$ 
for $s<t$. As in the proof of Theorem \ref{teo2_intro}, it follows from the minimizing property of $\Sigma$ that $H(t)\equiv H$ and $J(t)= J(0)$ for all $t\in(-\varepsilon,\varepsilon)$. Thus, $\Sigma_{t}$ are all capillary $cmc$ stable hypersurfaces. Applying Proposition \ref{prop6} to $\Sigma_t$ we conclude that they are also infinitesimally rigid. 

It follows from \eqref{lapse_cor} that $\rho_t$ satisfies the following homogeneous Neumann problem
$$
\left\lbrace\begin{array}{rl}
\Delta_{\Sigma}\rho_t=0     & \text{in} \ \Sigma_{t},  \\[0.2cm]
\frac{\partial \rho_t}{\partial \nu_{t}}=0  & \text{in} \ \partial\Sigma_{t}.
\end{array}\right.
$$
 Therefore $\rho_{t}$ is constant, and without loss of generality, let us assume $\rho_t = 1$. Up to an isometry, we can write the metric of $M$, in a sufficiently small neighborhood of $\Sigma$, as $\overline{g} = dt^{2} + g_{t}$. Since $\Sigma_t$ are totally umbilical by \cite[Lem.7.4]{GF} the metric $g_t$ has the following evolution equation
	\begin{eqnarray*}
		\frac{\partial}{\partial t}g_{ij}(t)
		&=&-2H\rho_{t}g_{ij}(t).
	\end{eqnarray*}
Therefore, $g_t=e^{-2Ht}g$ for every $t \in (-\epsilon,\epsilon)$. This concludes the proof. 
\end{proof}

\bibliographystyle{amsplain}

\begin{thebibliography}{99.}

%
%




\bibitem{Ambrozio} Ambrozio, L.C. \emph{Rigidity of area-minimizing free boundary surfaces in mean convex three-manifolds.} J. Geom. Anal. \textbf{25} (2015) n.2, 1001--1017.

\bibitem{Barbosa} Barbosa, E. \emph{On CMC free-boundary stable hypersurfaces in a Euclidean ball.} Math. Ann. \textbf{372} (2018) n.1-2, 179--187.

\bibitem{BdCE} Barbosa, J.L., do Carmo, M., Eschenburg, J. \emph{Stability of hypersurfaces of constant mean curvature in Riemannian manifolds.} Math. Z. \textbf{197} (1988) n.1, 123--138.

\bibitem{Barros} Barros, A., Cruz, C. \emph{Free boundary hypersurfaces with nonpositive Yamabe invariant in mean convex manifolds.} J. Geom. Anal. \textbf{30} (2020) n.4, 3542--3562.

\bibitem{BBN} Bray, H., Brendle, S., Neves, A. \emph{Rigidity of area-minimizing two-spheres in three-manifolds.} Commun. Anal. Geom. \textbf{18} (2010) n.4, 821--830.


\bibitem{Br} Brendle, S. \emph{A sharp bound for the area of minimal surfaces in the unit ball.} Geom. Funct. Anal. \textbf{22} (2012) n.3, 621--626.

\bibitem{Cai} Cai, M. \emph{Volume minimizing hypersurfaces in manifolds of nonnegative scalar curvature}, Minimal surfaces, geometric analysis and symplectic geometry (Baltimore, MD, 1999), Adv. Stud. Pure Math., vol. 34, Math. Soc. Japan, Tokyo, pp. 1--7 (2002).

\bibitem{CG} Cai, M., Galloway, G. \emph{Rigidity of area-minimizing tori in 3-manifolds of nonnegative scalar curvature.} Commun. Anal. Geom. \textbf{8} (2000) n.3, 565--573.

\bibitem{CFS} Carlotto, A., Franz, G., Schulz, M. \emph{Free boundary minimal surfaces with connected boundary and arbitrary genus.} Camb. J. Math. \textbf{10} (2022) n.4, 835--857.

\bibitem{CFP} Chen, J., Fraser, A., Pang, C. \emph{Minimal immersions of compact bordered Riemann surfaces with free boundary.} Trans. Amer. Math. Soc. \textbf{367} (2015) n.4, 2487--2507.


\bibitem{almeida_mendes} de Almeida, D., Mendes, A. \emph{Rigidity results for free boundary hypersurfaces in initial data sets with boundary.} Personal communication, (2024).

\bibitem{EscobarU} Escobar, J. \emph{Uniqueness theorems on conformal deformation of metrics, Sobolev inequalities and an eigenvalue estimate.} Commun. Pure Appl. Math. \textbf{43} (1990), 857–883.

\bibitem{EscobarT}  Escobar, J. \emph{The Yamabe problem on manifolds with boundary.} J. Differential Geom. \textbf{35} (1992), 21-–84.


\bibitem{Finn} Finn, R. \emph{Equilibrium capillary surfaces}. Grundlehren Math. Wiss., 284 [Fundamental Principles of Mathematical Sciences] Springer-Verlag, New York, 1986. xvi+245 pp.


\bibitem{FPZ} Folha, A., Pacard, F., Zolotareva, T. \emph{Free boundary minimal surfaces in the unit 3-ball.} manuscripta math. \textbf{154} (2017), 359--409. 


\bibitem{FL} Fraser, A., Li, M. \emph{Compactness of the space of embedded minimal surfaces with free boundary in three-manifolds with nonnegative Ricci curvature and convex boundary.} J. Differ. Geom. \textbf{96} (2014) n.2, 183--200.


\bibitem{FS} Fraser, A., Schoen, R. \emph{Steklov eigenvalue, conformal geometry and minimal surfaces.} Adv. Math. \textbf{226} (2011) n.5, 4011--4030.

\bibitem{FS1} Fraser, A., Schoen, R. \emph{Sharp eigenvalue bounds and minimal surfaces in the ball.} Invent. Math. \textbf{203} (2016), 823--890.

\bibitem{gronwall} Gronwall, T.H. \emph{Note on the Derivatives with Respect to a Parameter of the Solutions of a System of Differential Equations.} Ann. of Math. (2)  \textbf{20} (1919) n.4, 292--296.


\bibitem{hong_saturnino} Hong, H., Saturnino, A.B. \emph{Capillary surfaces: stability, index and curvature estimates.} J. Reine Angew. Math. \textbf{803} (2023), 233--265.

\bibitem{GF} Huisken, G., Polden, A. \emph{Geometric evolution equations for hypersurfaces.} In: Hildebrandt, S., Struwe, M.(eds) Calculus of variations and geometric evolution problems. (2006) 45--84, Lecture Notes in Math. vol. 1713.


\bibitem{Li1} Li, M. \emph{A general existence theorem of embedded minimal surfaces with free boundary.} Comm. Pure Appl. Math. \textbf{68} (2015), 286--331.

\bibitem{Li2} Li, M. \emph{Free boundary minimal surfaces in the unit ball: recent advances and open questions}, Proceedings of the International Consortium of Chinese Mathematicians 2017, Int. Press, Boston, MA (2020), pp. 401--435.

\bibitem{Li} Li, C. \emph{A polyhedron comparison theorem for 3-manifolds with positive scalar curvature.} Invent. Math. \textbf{219} (2020), 1--37.

\bibitem{LM} Lima, V., Menezes, A. \emph{A two-piece property for free boundary minimal surfaces in the ball.} Trans. Amer. Math. Soc. \textbf{374} (2021) n.3, 1661--1686.

\bibitem{Longa} Longa, E.R., \emph{Low Index Capillary Minimal Surfaces in Riemannian 3-Manifolds}. J. Geom. Anal. \textbf{32} (2022) n. 143, 1--21.

\bibitem{MNS} M\'aximo, D., Nunes, I., Smith, G. \emph{Free boundary minimal annuli in convex three-manifolds.}  J. Differential Geom. \textbf{106} 2017 n.1, 139--186.

\bibitem{MM} Micallef, M., Moraru, V. \emph{Splitting of 3-Manifolds and rigidity of area-minimizing surfaces.} Proc. Amer. Math. Soc. \textbf{143} (2015) n.7, 2865--2872.

\bibitem{Moraru25} Moraru, V. \emph{On area comparison and rigidity involving the scalar curvature.} J. Geom. Anal. \textbf{26} (2016) n.1, 294--312.

\bibitem{Ni} Nitsche, J.C.C. \emph{Stationary partitioning of convex bodies.}Arch. Ration. Mech. Anal. \textbf{89} (1985), 1–19.

\bibitem{Nunes1} Nunes, I. \emph{Rigidity of area-minimizing hyperbolic surfaces in three-manifolds.} J. Geom. Anal. \textbf{23} (2013) n.3, 1290--1302.

\bibitem{Nunes} Nunes, I. \emph{On stable constant mean curvature surfaces with free boundary.} Math. Z. \textbf{287} (2017) n.1-2, 473--479.

\bibitem{Ros} Ros, A., Souam, R. \emph{On stability of capillary surfaces in a ball.} Pac. J. Math. \textbf{178} (1997) n.2, 345–-361.

\bibitem{RV} Ros, A., Vergasta, E. \emph{Stability for hypersurfaces of constant mean curvature with free boundary.} Geom. Dedicata \textbf{56} (1995), 19--33.


\bibitem{Schoen89} Schoen, R. \emph{Variational theory for the total scalar curvature functional for riemannian metrics and related topics.} In: Giaquinta, M. (eds) Topics in Calculus of Variations. Lecture Notes in Mathematics, vol 1365. Springer, Berlin, Heidelberg (1989).

\bibitem{Wang-Xia} Wang, G., Xia, C. \emph{Uniqueness of stable capillary hypersurfaces in a ball.} Math. Ann. \textbf{374} (2019) n.3-4, 1845--1882.

\bibitem{Young} Young, T. \emph{An essay on the cohesion of fluids}, Philos. Trans. Roy. Soc. London \textbf{95} (1805), 65--87. 
%
%
%
%
%
%
%
%
%
%
%
%
%
%
%
%
%
%
%
%
%
%
%


\end{thebibliography}

\end{document}